\documentclass[12pt, reqno]{amsart}

\usepackage[T1]{fontenc}
\usepackage[utf8]{inputenc}
\usepackage{lmodern}
\usepackage{amsmath, amsthm, amscd, amsfonts, amssymb, graphicx, color}
\usepackage[bookmarksnumbered, colorlinks, plainpages]{hyperref}

\newtheorem{theorem}{Theorem}[section]

\newtheorem{proposition}[theorem]{Proposition}
\newtheorem{corollary}[theorem]{Corollary}

\theoremstyle{definition}

\newtheorem{example}[theorem]{Example}

\theoremstyle{remark}

\numberwithin{equation}{section}

\begin{document}
	
	\setcounter{page}{1}
	
	
	\title[An Invertible Hurwitz--Lerch Type Family]
	{An Invertible Family of Hurwitz--Lerch Type Functions Associated with
		\(\boldsymbol{k}\)-Augmented Centered Triangular Numbers}
	
	
	\author[N. Lacpao, R. Acope, M. Frias, M. Arcillas, R. Tiu, F. Romagos]
	{Noel B. Lacpao$^1$$^{*}$, Rushel S. Acope$^2$, Marlon S. Frias$^3$,
		Mark Ivan P. Arcillas$^4$, Rey Carl P. Tiu$^5$,
		\MakeLowercase{and} Francis Jay R. Romagos$^6$}
	
	\address{$^{1}$ Department of Mathematics, Bukidnon State University, Malaybalay City, Philippines.}
	\email{\textcolor[rgb]{0.00,0.00,0.84}{noel.lacpao@buksu.edu.ph}}
	
	\address{$^{2}$ Department of Mathematics, Bukidnon State University, Malaybalay City, Philippines.}
	\email{\textcolor[rgb]{0.00,0.00,0.84}{rushelacope16@gmail.com}}
	
	\address{$^{3}$ Department of Mathematics, Bukidnon State University, Malaybalay City, Philippines.}
	\email{\textcolor[rgb]{0.00,0.00,0.84}{marlonfrias@buksu.edu.ph}}
	
	\address{$^{4}$ Department of Mathematics, Bukidnon State University, Malaybalay City, Philippines.}
	\email{\textcolor[rgb]{0.00,0.00,0.84}{markivanarcillas@buksu.edu.ph}}
	
	\address{$^{5}$ Department of Mathematics, Bukidnon State University, Malaybalay City, Philippines.}
	\email{\textcolor[rgb]{0.00,0.00,0.84}{reycarl\_tiu@buksu.edu.ph}}
	
	\address{$^{6}$ Department of Mathematics, Bukidnon State University, Malaybalay City, Philippines.}
	\email{\textcolor[rgb]{0.00,0.00,0.84}{romagosfrancisjay@gmail.com}}
	
	
	\begin{abstract}
		This paper defines a family of Hurwitz--Lerch type functions whose
		coefficients are the \(k\)-augmented centered triangular numbers. For this
		family, we obtain the convergence conditions, a reduction formula, and an
		Euler-operator form. A Vandermonde-based inversion formula is derived for a
		class of polynomially weighted Hurwitz--Lerch functions. The family considered
		here is the quadratic case with geometric factors \(1\), \(2\), and \(4\).
		The resulting formulas show that three consecutive functions recover the
		classical Hurwitz--Lerch transcendent and its first two Euler derivatives.
		We also derive recurrence formulas, ordinary generating functions, finite
		sums, and special values. The values at \(z=1\) are expressed through
		Hurwitz zeta functions and Bernoulli polynomials. When \(a=1\), the numerator
		polynomials of the rational values \(H_k(z,-m,1)\) are written in terms of
		Eulerian polynomials, while the alternating values \(H_k(-1,-m,a)\) are
		expressed through Euler polynomials.
		
		\medskip
		
		\noindent
		\textit{Keywords.}
		Hurwitz--Lerch transcendent, \(k\)-augmented centered triangular numbers,
		Euler operator, Vandermonde inversion, Bernoulli polynomials, Eulerian
		polynomials, Euler polynomials.
		
		\smallskip
		
		\noindent
		\textit{2020 Mathematics Subject Classification.}
		Primary 11M35; Secondary 11B83, 05A15.
	\end{abstract}
	
	\maketitle
	
	\section{Introduction}
	
	The Hurwitz--Lerch transcendent is defined by
	\begin{equation}\label{eq:hurwitz-lerch}
		\Phi(z,s,a)	=\sum_{r=0}^{\infty}\frac{z^r}{(r+a)^s},
	\end{equation}
	where the parameters are chosen so that the series converges. It includes the
	Hurwitz zeta function and the polylogarithm through
	\[
	\Phi(1,s,a)=\zeta(s,a)
	\]
	and
	\[
	\Phi(z,s,1)
	=
	\frac{\operatorname{Li}_s(z)}{z},
	\qquad z\neq0.
	\]
	These relations and the standard analytic properties of the function are
	discussed in \cite{NIST2026,SrivastavaChoi2012}. The early development of the
	Hurwitz and Lerch zeta functions may be traced through
	\cite{Apostol1951,Hurwitz1882,Lerch1887}. For a detailed treatment of the Lerch
	zeta function, see \cite{LaurincikasGarunkstis2002}.
	
	Several extensions of the Hurwitz--Lerch transcendent have been obtained by
	introducing extra parameters, coefficient factors, or auxiliary kernels.
	Choi and Parmar \cite{ChoiParmar2017} derived integral representations,
	summation formulas, and Mellin--Barnes type contour representations for a
	two-variable extension. Nisar \cite{Nisar2019} studied a further extension and
	obtained integral representations, a summation formula, and relations with
	generalized hypergeometric functions.
	
	Integral formulas, generating relations, and derivative formulas for related
	Hurwitz--Lerch type functions were obtained in
	\cite{ChoiSahinYagciKim2020,NadeemUsmanNisarBaleanu2020}. Recurrence relations
	for a two-variable extension were also derived in
	\cite{ChoiSahinYagciKim2020}. Geometric properties involving differential
	subordination, univalence, and convexity were studied in
	\cite{Abdulnabi2024}. General families of Hurwitz--Lerch functions and their
	representations are discussed in \cite{Srivastava2019}.
	
	Inversion formulas have also appeared in the literature. Nakamura
	\cite{Nakamura2010} obtained inversion formulas from multiplication relations
	for Hurwitz--Lerch zeta functions. Bayad and Chikhi
	\cite{BayadChikhi2013} studied reduction and duality formulas for generalized
	Hurwitz--Lerch zeta functions. The inversion considered here has a different
	source because it acts on an augmentation parameter and separates distinct geometric
	components in that parameter.
	
	The coefficient sequence used in this paper is formed from the
	\(k\)-augmented centered triangular numbers introduced in
	\cite{Lacpao2026}. Let \(A(n,k)\) denote the \(n\)-th number in this family.
	We define
	\begin{equation}
		H_k(z,s,a)
		=
		\sum_{r=0}^{\infty}
		\frac{A(r+1,k)z^r}{(r+a)^s},
		\label{hk-definition}
	\end{equation}
	where \(k\in\mathbb{N}_0\),
	\(a\notin\{0,-1,-2,\ldots\}\), and \(z,s\in\mathbb{C}\), subject to the
	convergence conditions established below. The closed form of \(A(r+1,k)\) is quadratic in \(r\)
	and its dependence on \(k\) is determined by the three geometric sequences \(1^k\), \(2^k\), and \(4^k\). As a result, \(H_k(z,s,a)\) is a linear combination of \(\Phi(z,s,a)\), \(D\Phi(z,s,a)\), and \(D^2\Phi(z,s,a)\) where \(D=z\frac{d}{dz}\).	Therefore, the three consecutive functions \(H_k(z,s,a)\), \(H_{k+1}(z,s,a)\), and
	\(H_{k+2}(z,s,a)\) determine the Hurwitz--Lerch transcendent and its first two Euler derivatives.
	
	The inversion is obtained as the quadratic case of a Vandermonde formula for
	polynomially weighted Hurwitz--Lerch functions. For the family considered here,
	the geometric factors are \(1\), \(2\), and \(4\).
		
	\section{Preliminaries}
	
	Throughout the paper, let \(k\in\mathbb{N}_0\),
	\(a\notin\{0,-1,-2,\ldots\}\), and \(z,s\in\mathbb{C}\), unless stated
	otherwise. For complex powers, we use the principal value
	\[
	(r+a)^{-s}
	=
	\exp\{-s\operatorname{Log}(r+a)\},
	\]
	where
	\[
	\operatorname{Log}w
	=
	\log|w|+i\operatorname{Arg}w,
	\ \text{such that} \ 
	-\pi<\operatorname{Arg}w\leq\pi.
	\]
	
	The Hurwitz--Lerch transcendent is defined by
	\eqref{eq:hurwitz-lerch}. For \(|z|<1\), its defining series converges
	absolutely for every \(s\in\mathbb{C}\). When \(z=1\), it reduces to the
	Hurwitz zeta function, \(\Phi(1,s,a)=\zeta(s,a)\), whose defining series converges absolutely for	\(\operatorname{Re}(s)>1\). When \(a=1\) and \(z\neq 0\),
	\begin{equation*}
		\Phi(z,s,1)
		=
		\frac{\operatorname{Li}_s(z)}{z}.
	\end{equation*}
	These convergence properties and specializations are standard
	\cite{NIST2026,SrivastavaChoi2012}.
	
	The Hurwitz--Lerch transcendent has an analytic continuation beyond its
	initial domain of convergence. Detailed accounts of its continuation and
	functional properties are given in \cite{LagariasLi2012,LaurincikasGarunkstis2002}.
	Let
	\begin{equation}\label{eq:EulerOperator}
		D=z\frac{d}{dz}
	\end{equation}
	denote the Euler operator. In a region where termwise differentiation is
	valid,
	\begin{equation}\label{euler}
		D\Phi(z,s,a)
		=
		\sum_{r=0}^{\infty}
		\frac{rz^r}{(r+a)^s}.
	\end{equation}
	Since \(r=(r+a)-a\), we obtain
	\begin{equation}
		D\Phi(z,s,a)
		=
		\Phi(z,s-1,a)-a\Phi(z,s,a).
		\label{eq:euler-first}
	\end{equation}
	Applying \(D\) once more and writing \(r^2=(r+a)^2-2a(r+a)+a^2\), we get
	\begin{align}
		D^2\Phi(z,s,a)
		&=\sum_{r=0}^{\infty}\frac{r^2z^r}{(r+a)^s}\label{euler2}\\
		&=\Phi(z,s-2,a)-2a\Phi(z,s-1,a)+a^2\Phi(z,s,a)\label{eq:euler-second}.
	\end{align}
	The identities \eqref{eq:euler-first} and \eqref{eq:euler-second} are first
	valid in a common domain of absolute convergence. They extend to other
	permitted values of the parameters by analytic continuation.
	
	For \(n\geq1\) and \(k\in\mathbb{N}_0\), the \(k\)-augmented centered triangular numbers have the closed form
	\begin{equation}
		A(n,k)
		=
		\frac{3}{2}4^k(n-1)^2
		+
		\frac{3}{2}2^k(n-1)
		+
		1.
		\label{eq:ank}
	\end{equation}
	This formula follows from the arithmetic structure of the
	\(k\)-augmented centered triangular array
	\cite{Lacpao2026}. Taking \(n=r+1\) in \eqref{eq:ank} while setting \(\alpha_k=\frac{3}{2}4^k\) and \(\beta_k=\frac{3}{2}2^k\), we obtain
	\begin{equation}
		A(r+1,k)
		=
		\alpha_k r^2+\beta_k r+1.
		\label{eq:ark}
	\end{equation}
		
	The Hurwitz zeta function at nonpositive integers is related to the Bernoulli
	polynomials by
	\begin{equation}
		\zeta(-m,a)
		=
		-\frac{B_{m+1}(a)}{m+1},
		\qquad
		m\in\mathbb{N}_0.
		\label{eq:zeta-bernoulli}
	\end{equation}
	This identity will be used to evaluate \(H_k(1,-m,a)\)
	\cite{NIST2026,SrivastavaChoi2012}.
	
	\section{The Weighted Hurwitz--Lerch Function}
	
	We begin with the convergence of the family defined in \eqref{hk-definition}.
	
	\begin{theorem}
		\label{thm:convergence}
		Let \(k\in\mathbb{N}_0\) and
		\(a\notin\{0,-1,-2,\ldots\}\). The defining series for \(H_k(z,s,a)\) converges absolutely for every \(s\in\mathbb{C}\) when \(|z|<1\). If \(z=1\), then the series converges if and only if \(\operatorname{Re}(s)>3\) and the convergence is absolute.	If \(|z|=1\) and \(z\neq1\), then the series converges if and only if \(\operatorname{Re}(s)>2\). In this case, the convergence is absolute for \(\operatorname{Re}(s)>3\) and conditional for \(2<\operatorname{Re}(s)\leq3\).
	\end{theorem}
	
	\begin{proof}
		We write \(\sigma=\operatorname{Re}(s)\). By \eqref{eq:ark}, we have
		\[
		A(r+1,k)
		=
		\alpha_k r^2
		\left(
		1+O\left(\frac1r\right)
		\right).
		\]
		For fixed \(a\) and \(s\),
		\begin{align*}
			(r+a)^{-s}
			&=
			r^{-s}
			\left(
			1+\frac{a}{r}
			\right)^{-s}\\
			&=
			r^{-s}
			\left(
			1+O\left(\frac1r\right)
			\right)
		\end{align*}
		as \(r\to\infty\). Therefore,
		\begin{equation}
			\frac{A(r+1,k)}{(r+a)^s}
			=
			\alpha_k r^{2-s}
			\left(
			1+O\left(\frac1r\right)
			\right).
			\label{eq:hk-term-asymptotic}
		\end{equation}
		In particular,
		\[
		\left|
		\frac{A(r+1,k)}{(r+a)^s}
		\right|
		\asymp
		r^{2-\sigma}.
		\]
		
		When \(|z|<1\), the factor \(|z|^r\) gives exponential decay. Hence, the
		series converges absolutely for every \(s\in\mathbb{C}\). Suppose that \(|z|=1\). Absolute convergence is determined by \(\displaystyle\sum_{r=1}^{\infty}r^{2-\sigma}\)
		which converges if and only if \(\sigma>3\). Now let \(z=1\). If \(\sigma\leq2\), then the general term does not
		converge to zero. If \(2<\sigma\leq3\), then \eqref{eq:hk-term-asymptotic} gives
		\begin{equation*}
			\frac{A(r+1,k)}{(r+a)^s}
		=
		\alpha_k r^{2-s}
		+
		O\left(r^{1-\sigma}\right).
		\end{equation*}
		The series formed by the error terms converges absolutely because
		\(\sigma>2\). On the other hand, the Dirichlet series \(\displaystyle\sum_{r=1}^{\infty}r^{-(s-2)}\)
		converges if and only if \( \operatorname{Re}(s-2)>1\), and therefore diverges when \(\sigma\leq3\). It follows that the series
		for \(z=1\) converges if and only if \(\sigma>3\).
		
		Finally, suppose that \(|z|=1\) and \(z\neq1\). Set
		\[
		b_r
		=
		\frac{A(r+1,k)}{(r+a)^s}.
		\]
		Assume first that \(\sigma>2\). Then
		\[
		b_r
		=
		O\left(r^{2-\sigma}\right),
		\]
		so \(b_r\to0\). Choose \(R>0\) sufficiently large so that \(\operatorname{Re}(x+a)>0\)
		for every \(x\geq R\), and define
		\[
		f(x)
		=
		\frac{\alpha_kx^2+\beta_kx+1}{(x+a)^s},
		\ x\geq R.
		\]
		Then
		\[
		f'(x)
		=
		(2\alpha_kx+\beta_k)(x+a)^{-s}
		-
		s(\alpha_kx^2+\beta_kx+1)(x+a)^{-s-1}.
		\]
		As \(x\to\infty\),
		\[
		|(x+a)^{-s}|
		=
		O\left(x^{-\sigma}\right)
		\]
		and
		\[
		|(x+a)^{-s-1}|
		=
		O\left(x^{-\sigma-1}\right).
		\]
		Hence,
		\[
		f'(x)
		=
		O\left(x^{1-\sigma}\right).
		\]
		For every sufficiently large integer \(r\),
		\[
		b_{r+1}-b_r
		=
		f(r+1)-f(r)
		=
		\int_r^{r+1}f'(x)\,dx,
		\]
		so
		\[
		b_{r+1}-b_r
		=
		O\left(r^{1-\sigma}\right).
		\]
		Since \(\sigma>2\),
		\[
		\sum_{r=0}^{\infty}
		|b_{r+1}-b_r|
		\]
		converges. The partial sums
		\[
		\sum_{r=0}^{N}z^r
		=
		\frac{1-z^{N+1}}{1-z}
		\]
		are bounded because \(z\neq1\). Hence, summation by parts implies that \(\sum_{r=0}^{\infty}b_rz^r\) converges.
		
		If \(\sigma\leq2\), then \(|b_r|\asymp r^{2-\sigma},\) so the general term \(b_rz^r\) does not converge to zero. Therefore, the
		series converges for \(|z|=1\), \(z\neq1\), if and only if
		\(\sigma>2\). The convergence is absolute for \(\sigma>3\) and
		conditional when \(2<\sigma\leq3.\)
	\end{proof}
	\begin{theorem}
		\label{thm:reduction}
		The Hurwitz--Lerch type function associated with the
		\(k\)-augmented centered triangular numbers satisfies
		\begin{align}
			H_k(z,s,a)
			&=
			\alpha_k\Phi(z,s-2,a)
			+
			(\beta_k-2a\alpha_k)\Phi(z,s-1,a)
			\notag\\
			&\quad+
			(1-a\beta_k+a^2\alpha_k)\Phi(z,s,a).
			\label{eq:hk-reduction-main}
		\end{align}
	\end{theorem}
		
	\begin{proof}
		By \eqref{eq:ark},
		\begin{equation*}
			H_k(z,s,a)=\alpha_k\sum_{r=0}^{\infty}\frac{r^2z^r}{(r+a)^s}+\beta_k\sum_{r=0}^{\infty}
			\frac{rz^r}{(r+a)^s}+\Phi(z,s,a).
		\end{equation*}
		Using \eqref{eq:euler-first} and \eqref{eq:euler-second}, we obtain
		\begin{align*}
			H_k(z,s,a)
			&=
			\alpha_k
			\bigl[
			\Phi(z,s-2,a)
			-
			2a\Phi(z,s-1,a)
			+
			a^2\Phi(z,s,a)
			\bigr]
			\\
			&\quad+
			\beta_k
			\bigl[
			\Phi(z,s-1,a)
			-
			a\Phi(z,s,a)
			\bigr]
			+
			\Phi(z,s,a).
		\end{align*}
		Collecting the coefficients of the three Hurwitz--Lerch functions gives
		\eqref{eq:hk-reduction-main}.
	\end{proof}
	
	Formula \eqref{eq:hk-reduction-main} expresses \(H_k(z,s,a)\) as a linear
	combination of \(\Phi(z,s-2,a)\), \(\Phi(z,s-1,a)\), and \(\Phi(z,s,a)\). The dependence on \(k\) occurs through \(\alpha_k\) and \(\beta_k\).
	
	\section{Examples and Initial Cases}
	
	For \(k=0,1,2\), the coefficients \(\alpha_k\) and \(\beta_k\) are
	\begin{equation*}
		(\alpha_0,\beta_0)
		=
		\left(\frac{3}{2},\frac{3}{2}\right),
		\
		(\alpha_1,\beta_1)
		=
		(6,3),
		\ \text{and} \
		(\alpha_2,\beta_2)
		=
		(24,6).
	\end{equation*}
	Hence, we have
	\begin{enumerate}
		\item[(i)] \(A(r+1,0)=\frac{3}{2}r^2+\frac{3}{2}r+1,\)
		\item[(ii)] \(A(r+1,1)=6r^2+3r+1,\) and 
		\item[(iii)] \(A(r+1,2)=24r^2+6r+1.\)
	\end{enumerate}
	By Theorem \ref{thm:reduction}, we obtain
	\begin{align*}
		H_0(z,s,a)&=\frac{3}{2}\Phi(z,s-2,a)+\left(\frac{3}{2}-3a\right)\Phi(z,s-1,a)+
		\left(1-\frac{3}{2}a+\frac{3}{2}a^2\right)\Phi(z,s,a),\\
		H_1(z,s,a)&=6\Phi(z,s-2,a)+(3-12a)\Phi(z,s-1,a)	+(1-3a+6a^2)\Phi(z,s,a),\ \text{and}\\
		H_2(z,s,a)&=24\Phi(z,s-2,a)+(6-48a)\Phi(z,s-1,a)+(1-6a+24a^2)\Phi(z,s,a).
	\end{align*}
		
	In terms of the Euler operator, these formulas become
	\[
	H_0(z,s,a)
	=
	\left(
	\frac{3}{2}D^2+\frac{3}{2}D+1
	\right)\Phi(z,s,a),
	\]
	\[
	H_1(z,s,a)
	=
	\left(
	6D^2+3D+1
	\right)\Phi(z,s,a),
	\]
	and
	\[
	H_2(z,s,a)
	=
	\left(
	24D^2+6D+1
	\right)\Phi(z,s,a).
	\]
	
	Solving the resulting linear system gives
	\[
	\Phi(z,s,a)
	=
	\frac{8}{3}H_0(z,s,a)
	-
	2H_1(z,s,a)
	+
	\frac{1}{3}H_2(z,s,a),
	\]
	\[
	D\Phi(z,s,a)
	=
	\frac{
		-4H_0(z,s,a)
		+
		5H_1(z,s,a)
		-
		H_2(z,s,a)
	}{3},
	\]
	and
	\[
	D^2\Phi(z,s,a)
	=
	\frac{
		2H_0(z,s,a)
		-
		3H_1(z,s,a)
		+
		H_2(z,s,a)
	}{9}.
	\]
	These identities are the case \(k=0\) of
	Corollary~\ref{cor:hk-vandermonde-inversion}.
	
	\section{Operator Representation and Inversion Formulas}
	
	The family \(H_k(z,s,a)\) is the quadratic case associated with the geometric
	factors \(1\), \(2\), and \(4\).
	
	The identities are first proved in a common domain of absolute convergence.
	They may be extended by analytic continuation whenever the functions involved
	are defined.
	
	\begin{proposition}
		\label{prop:operator-form}
		The Hurwitz--Lerch type function weighted by the
		\(k\)-augmented centered triangular numbers satisfies
		\begin{equation}
			H_k(z,s,a)
			=
			\left(
			\alpha_kD^2+\beta_kD+1
			\right)\Phi(z,s,a).
			\label{eq:operator-form}
		\end{equation}
	\end{proposition}
	\begin{proof}
		By \eqref{eq:ark}, \eqref{euler}, and \eqref{euler2},
		\begin{align*}
			\left(\alpha_kD^2+\beta_kD+1\right)\Phi(z,s,a)
			&=
			\sum_{r=0}^{\infty}
			\frac{(\alpha_kr^2+\beta_kr+1)z^r}{(r+a)^s}\\
			&=
			H_k(z,s,a).
		\end{align*}
	\end{proof}
		
	The following theorem gives the Vandermonde inversion formula for a class of polynomially
	weighted Hurwitz--Lerch functions.
	
	\begin{theorem}
		\label{thm:vandermonde-inversion}
		Let \(d\in\mathbb{N}_0\), and let \(\lambda_0,\lambda_1,\ldots,\lambda_d\) be distinct nonzero complex numbers. Let \(c_0,c_1,\ldots,c_d\)	be nonzero complex constants. For \(k\in\mathbb{N}_0\), define
		\[
		W_k(r)
		=
		\sum_{j=0}^{d}
		c_j\lambda_j^k r^j
		\]
		and
		\begin{equation}
			\mathcal{F}_k(z,s,a)
			=
			\sum_{r=0}^{\infty}
			\frac{W_k(r)z^r}{(r+a)^s},
			\label{eq:general-weighted-lerch}
		\end{equation}
		where \(a\notin\{0,-1,-2,\ldots\}\), and the parameters are chosen so that
		the series converges absolutely. For each \(j=0,1,\ldots,d\), let
		\begin{equation}
			L_j(x)
			=
			\prod_{\substack{0\leq m\leq d\\m\neq j}}
			\frac{x-\lambda_m}{\lambda_j-\lambda_m}
			=
			\sum_{q=0}^{d}
			\ell_{j,q}x^q.
			\label{eq:lagrange-polynomials}
		\end{equation}
		Then
		\begin{equation}
			D^j\Phi(z,s,a)
			=
			\frac{1}{c_j\lambda_j^k}
			\sum_{q=0}^{d}
			\ell_{j,q}\mathcal{F}_{k+q}(z,s,a).
			\label{eq:general-vandermonde-inversion}
		\end{equation}
		Consequently, the \(d+1\) consecutive functions \(\mathcal{F}_k, \mathcal{F}_{k+1}, \ldots,		\mathcal{F}_{k+d}\)	determine \(\Phi, D\Phi, \ldots, D^d\Phi.\)
	\end{theorem}
	
	\begin{proof}
		For \(j=0,1,\ldots,d\), termwise application of the Euler operator gives
		\[
		D^j\Phi(z,s,a)
		=
		\sum_{r=0}^{\infty}
		\frac{r^jz^r}{(r+a)^s},
		\]
		where \(D^0\Phi=\Phi\). It follows from
		\eqref{eq:general-weighted-lerch} that
		\begin{equation}
			\mathcal{F}_k(z,s,a)
			=
			\sum_{j=0}^{d}
			c_j\lambda_j^kD^j\Phi(z,s,a).
			\label{eq:general-operator-form}
		\end{equation}
		For \(q=0,1,\ldots,d\), we replace \(k\) by \(k+q\) in
		\eqref{eq:general-operator-form} to get
		\begin{equation*}
			\mathcal{F}_{k+q}(z,s,a)
			=
			\sum_{m=0}^{d}
			c_m\lambda_m^{k+q}D^m\Phi(z,s,a).
		\end{equation*}
		Now, fix \(j\in\{0,1,\ldots,d\}\). Then
		\begin{align*}
			\sum_{q=0}^{d}
			\ell_{j,q}\mathcal{F}_{k+q}(z,s,a)
			&=
			\sum_{q=0}^{d}
			\ell_{j,q}
			\sum_{m=0}^{d}
			c_m\lambda_m^{k+q}D^m\Phi(z,s,a)
			\\
			&=
			\sum_{m=0}^{d}
			c_m\lambda_m^k
			\left(
			\sum_{q=0}^{d}
			\ell_{j,q}\lambda_m^q
			\right)
			D^m\Phi(z,s,a)
			\\
			&=
			\sum_{m=0}^{d}
			c_m\lambda_m^k
			L_j(\lambda_m)
			D^m\Phi(z,s,a).
		\end{align*}
		By the definition of \(L_j\) in	\eqref{eq:lagrange-polynomials},
		\begin{equation*}
			L_j(\lambda_m)
			=
			\begin{cases}
				1, & m=j,\\
				0, & m\neq j.
			\end{cases}
		\end{equation*}
		Hence,
		\[
		\sum_{q=0}^{d}
		\ell_{j,q}\mathcal{F}_{k+q}(z,s,a)
		=
		c_j\lambda_j^kD^j\Phi(z,s,a).
		\]
		Division by \(c_j\lambda_j^k\) gives \eqref{eq:general-vandermonde-inversion}.
	\end{proof}
	
	The coefficient matrix of the system may be written as
	\[
	M_k
	=
	V\operatorname{diag}
	\left(
	c_0\lambda_0^k,
	c_1\lambda_1^k,
	\ldots,
	c_d\lambda_d^k
	\right),
	\]
	where
	\[
	V
	=
	\begin{pmatrix}
		1&1&\cdots&1\\
		\lambda_0&\lambda_1&\cdots&\lambda_d\\
		\lambda_0^2&\lambda_1^2&\cdots&\lambda_d^2\\
		\vdots&\vdots&&\vdots\\
		\lambda_0^d&\lambda_1^d&\cdots&\lambda_d^d
	\end{pmatrix}.
	\]
	Hence,
	\[
	\det M_k
	=
	\left(
	\prod_{j=0}^{d}c_j\lambda_j^k
	\right)
	\prod_{0\leq i<j\leq d}
	(\lambda_j-\lambda_i).
	\]
	Since the constants \(c_j\) and \(\lambda_j\) are nonzero and the numbers
	\(\lambda_0,\lambda_1,\ldots,\lambda_d\) are distinct, the determinant is
	nonzero.
	
	The weighted family \(H_k(z,s,a)\) is obtained from
	Theorem~\ref{thm:vandermonde-inversion} by taking \(d=2\). Indeed,
	\begin{equation}
		A(r+1,k)
		=
		1\cdot1^k r^0
		+
		\frac{3}{2}2^k r
		+
		\frac{3}{2}4^k r^2.
		\label{eq:hk-spectral-weight}
	\end{equation}
	Thus, \((\lambda_0,\lambda_1,\lambda_2)=(1,2,4)\) and \((c_0,c_1,c_2)=\left(1,\frac{3}{2},\frac{3}{2}\right).\)
	With these choices,
	\[
	W_k(r)=A(r+1,k)
	\]
	and
	\[
	\mathcal{F}_k(z,s,a)=H_k(z,s,a).
	\]
	
	\begin{corollary}
		\label{cor:hk-vandermonde-inversion}
		For every \(k\in\mathbb{N}_0\), the following inversion formulas hold:
		\begin{align}
			\Phi(z,s,a)&=\frac{8}{3}H_k(z,s,a)-2H_{k+1}(z,s,a)+\frac{1}{3}H_{k+2}(z,s,a),
			\label{eq:inversion-phi}\\
			D\Phi(z,s,a)&=\frac{-4H_k(z,s,a)+5H_{k+1}(z,s,a)-H_{k+2}(z,s,a)}{			3\cdot2^k}, \ \text{and}\label{eq:inversion-dphi}\\
			D^2\Phi(z,s,a)&=\frac{2H_k(z,s,a)-3H_{k+1}(z,s,a)+H_{k+2}(z,s,a)}{		9\cdot4^k}.
			\label{eq:inversion-d2phi}
		\end{align}

	\end{corollary}
	
	\begin{proof}
		For the geometric factors \(\lambda_0=1\), \(\lambda_1=2\), and \(\lambda_2=4\), the Lagrange polynomials in \eqref{eq:lagrange-polynomials} are
		\begin{align*}
			L_0(x)
			&=
			\frac{(x-2)(x-4)}{(1-2)(1-4)}
			=
			\frac{x^2-6x+8}{3},
			\\
			L_1(x)
			&=
			\frac{(x-1)(x-4)}{(2-1)(2-4)}
			=
			-\frac{x^2-5x+4}{2}, \ \text{and}
			\\
			L_2(x) 
			&=
			\frac{(x-1)(x-2)}{(4-1)(4-2)}
			=
			\frac{x^2-3x+2}{6}.
		\end{align*}
		Hence,
		\begin{align*}
			(\ell_{0,0},\ell_{0,1},\ell_{0,2})&=\left(\frac{8}{3},-2,\frac{1}{3}\right),\\
			(\ell_{1,0},\ell_{1,1},\ell_{1,2})&=\left(-2,\frac{5}{2},-\frac{1}{2}\right), \ \text{and}\\
			(\ell_{2,0},\ell_{2,1},\ell_{2,2})&=\left(\frac{1}{3},-\frac{1}{2},\frac{1}{6}
			\right).
		\end{align*}
		For \(j=0\), Theorem \ref{thm:vandermonde-inversion} gives
		\eqref{eq:inversion-phi}. Applying Theorem \ref{thm:vandermonde-inversion} with \(j=1\) and \(j=2\)
		gives \eqref{eq:inversion-dphi} and \eqref{eq:inversion-d2phi},	respectively.
	\end{proof}
	
	Since these geometric factors are distinct, the three consecutive functions
	\(H_k\), \(H_{k+1}\), and \(H_{k+2}\) recover \(\Phi(z,s,a)\), \(D\Phi(z,s,a)\), and \(D^2\Phi(z,s,a)\).
			
	\section{Recurrence Relations and Generating Functions}
	
	From Proposition~\ref{prop:operator-form},
	\[
	H_k(z,s,a)
	=
	\frac{3}{2}4^kD^2\Phi(z,s,a)
	+
	\frac{3}{2}2^kD\Phi(z,s,a)
	+
	\Phi(z,s,a).
	\]
	
	\begin{theorem}
		\label{thm:recurrence}
		For every \(k\in\mathbb{N}_0\),
		\begin{equation}
			H_{k+3}(z,s,a)
			=
			7H_{k+2}(z,s,a)
			-
			14H_{k+1}(z,s,a)
			+
			8H_k(z,s,a).
			\label{eq:third-order-recurrence}
		\end{equation}
	\end{theorem}
	
	\begin{proof}
		The three component sequences have characteristic roots \(4\), \(2\), and
		\(1\). Hence, the characteristic polynomial is
		\[
		(\lambda-4)(\lambda-2)(\lambda-1)
		=
		\lambda^3-7\lambda^2+14\lambda-8.
		\]
		This implies that every linear combination of \(4^k\), \(2^k\), and \(1\) satisfies the stated
		recurrence.
	\end{proof}
	
	The recurrence in \(k\) comes from the geometric decomposition of the
	coefficients of the \(k\)-augmented centered triangular numbers. This differs
	from generalized Hurwitz--Lerch constructions where extra parameters are
	inserted into the kernel or denominator of the series
	\cite{ChoiSahinYagciKim2020,NadeemUsmanNisarBaleanu2020}.
	
	\begin{theorem}
		\label{thm:ordinary-generating-function}
		Let
		\[
		G(y;z,s,a)
		=
		\sum_{k=0}^{\infty}
		H_k(z,s,a)y^k.
		\]
		For \(|y|<\frac14\),
		\begin{align}
			G(y;z,s,a)
			&=
			\frac{
				\frac{3}{2}D^2\Phi(z,s,a)
			}{
				1-4y
			}
			+
			\frac{
				\frac{3}{2}D\Phi(z,s,a)
			}{
				1-2y
			}
			+
			\frac{
				\Phi(z,s,a)
			}{
				1-y
			}.
			\label{eq:ordinary-generating-operator}
		\end{align}
		Using \eqref{eq:euler-first} and \eqref{eq:euler-second}, this may also be
		written as
		\begin{align}
			G(y;z,s,a)
			&=
			\frac{
				\frac{3}{2}
				\left[
				\Phi(z,s-2,a)
				-
				2a\Phi(z,s-1,a)
				+
				a^2\Phi(z,s,a)
				\right]
			}{
				1-4y
			}
			\notag\\
			&\quad+
			\frac{
				\frac{3}{2}
				\left[
				\Phi(z,s-1,a)
				-
				a\Phi(z,s,a)
				\right]
			}{
				1-2y
			}
			+
			\frac{
				\Phi(z,s,a)
			}{
				1-y
			}.
			\label{eq:ordinary-generating-shifted}
		\end{align}
	\end{theorem}
	
	\begin{proof}
		Using the operator representation,
		\begin{align*}
			\sum_{k=0}^{\infty}
			H_k(z,s,a)y^k
			&=
			\frac{3}{2}D^2\Phi(z,s,a)
			\sum_{k=0}^{\infty}(4y)^k
			\\
			&\quad+
			\frac{3}{2}D\Phi(z,s,a)
			\sum_{k=0}^{\infty}(2y)^k
			+
			\Phi(z,s,a)
			\sum_{k=0}^{\infty}y^k.
		\end{align*}
		The geometric series give the stated formula. The second form follows by
		substituting the shifted expressions for \(D\Phi(z,s,a)\) and
		\(D^2\Phi(z,s,a)\).
	\end{proof}
	
	\begin{corollary}
		\label{cor:recurrence-generating-function}
		The generating function may also be written as
		\begin{align*}
			G(y;z,s,a)
			&=
			\frac{
				H_0(z,s,a)
				+
				\left[
				H_1(z,s,a)-7H_0(z,s,a)
				\right]y
			}{
				1-7y+14y^2-8y^3
			}
			\\
			&\quad+
			\frac{
				\left[
				H_2(z,s,a)
				-
				7H_1(z,s,a)
				+
				14H_0(z,s,a)
				\right]y^2
			}{
				1-7y+14y^2-8y^3
			}.
		\end{align*}
	\end{corollary}
	
	\begin{proof}
		This follows by multiplying \(G(y;z,s,a)\) by
		\[
		1-7y+14y^2-8y^3
		=
		(1-y)(1-2y)(1-4y)
		\]
		and using \eqref{eq:third-order-recurrence}.
	\end{proof}
	
	The inversion formulas also give
	\[
	H_{k+2}(z,s,a)
	-
	3H_{k+1}(z,s,a)
	+
	2H_k(z,s,a)
	=
	9\cdot4^kD^2\Phi(z,s,a),
	\]
	and
	\[
	-4H_k(z,s,a)
	+
	5H_{k+1}(z,s,a)
	-
	H_{k+2}(z,s,a)
	=
	3\cdot2^kD\Phi(z,s,a).
	\]
	These identities follow directly from
	\eqref{eq:inversion-dphi} and \eqref{eq:inversion-d2phi}.
	
	\begin{corollary}
		For \(N\in\mathbb{N}_0\),
		\begin{align}
			\sum_{k=0}^{N}
			H_k(z,s,a)
			&=
			\frac{1}{2}
			\left(
			4^{N+1}-1
			\right)
			D^2\Phi(z,s,a)
			\notag\\
			&\quad+
			\frac{3}{2}
			\left(
			2^{N+1}-1
			\right)
			D\Phi(z,s,a)
			+
			(N+1)\Phi(z,s,a).
			\label{eq:finite-sum-hk}
		\end{align}
	\end{corollary}
	
	\begin{proof}
		From the operator representation,
		\[
		H_k(z,s,a)
		=
		\frac{3}{2}4^kD^2\Phi(z,s,a)
		+
		\frac{3}{2}2^kD\Phi(z,s,a)
		+
		\Phi(z,s,a).
		\]
		Using
		\[
		\sum_{k=0}^{N}4^k
		=
		\frac{4^{N+1}-1}{3},
		\qquad
		\sum_{k=0}^{N}2^k
		=
		2^{N+1}-1,
		\]
		and
		\[
		\sum_{k=0}^{N}\Phi(z,s,a)
		=
		(N+1)\Phi(z,s,a),
		\]
		\eqref{eq:finite-sum-hk} follows.
	\end{proof}
	
	\section{Special Values}
	
We first consider the values at \(z=1\). Related Hurwitz--Lerch type
poly-Cauchy and poly-Bernoulli families have been studied in
\cite{LacpaoCorcinoVega2019,Lacpao2023}.
	
	\begin{theorem}
		\label{thm:hurwitz-zeta-value}
		For \(\operatorname{Re}(s)>3\),
		\begin{align}
			H_k(1,s,a)
			&=
			\alpha_k\zeta(s-2,a)
			+
			(\beta_k-2a\alpha_k)\zeta(s-1,a)
			\notag\\
			&\quad+
			(1-a\beta_k+a^2\alpha_k)\zeta(s,a).
			\label{eq:hk-hurwitz-zeta}
		\end{align}
		The right-hand side gives a meromorphic continuation of
		\(H_k(1,s,a)\) to \(s\in\mathbb{C}\). This continuation has at most
		simple poles at \(s=1\), \(s=2\), and \(s=3\). The corresponding residues are
		\(1-a\beta_k+a^2\alpha_k\), \(\beta_k-2a\alpha_k\), and \(\alpha_k\), respectively.
	\end{theorem}
	
	\begin{proof}
		Take \(z=1\) in Theorem \ref{thm:reduction} and use \(\Phi(1,s,a)=\zeta(s,a)\). The continuation follows from the meromorphic continuation of the Hurwitz	zeta function. Since the Hurwitz zeta function has residue \(1\) at its pole \(s=1\), the stated residues follow from the coefficients of \(\zeta(s,a)\), \(\zeta(s-1,a)\), and \(\zeta(s-2,a)\),  respectively.
	\end{proof}
	
	\begin{corollary}
		For \(a=1\),
		\begin{equation*}
			H_k(1,s,1)
			=
			\alpha_k\zeta(s-2)
			+
			(\beta_k-2\alpha_k)\zeta(s-1)
			+
			(1-\beta_k+\alpha_k)\zeta(s).
		\end{equation*}
		
	\end{corollary}
	
	\begin{proof}
		Set \(a=1\) in Theorem~\ref{thm:hurwitz-zeta-value}.
	\end{proof}
	
	\begin{theorem}
		For \(m\in\mathbb{N}_0\), the continued value at \(s=-m\) is
		\begin{align}
			H_k(1,-m,a)
			&=
			-\alpha_k
			\frac{B_{m+3}(a)}{m+3}
			-
			(\beta_k-2a\alpha_k)
			\frac{B_{m+2}(a)}{m+2}
			\notag\\
			&\quad-
			(1-a\beta_k+a^2\alpha_k)
			\frac{B_{m+1}(a)}{m+1}.
			\label{eq:hk-bernoulli}
		\end{align}
	\end{theorem}
	
	\begin{proof}
		Use the meromorphic continuation in
		Theorem~\ref{thm:hurwitz-zeta-value}, put \(s=-m\), and apply
		\eqref{eq:zeta-bernoulli}.
	\end{proof}
	
For \(m\in\mathbb{N}_0\) and \(|z|<1\), define
\begin{equation}
	P_{m,k}(z,a)
	=
	H_k(z,-m,a)
	=
	\sum_{r=0}^{\infty}
	A(r+1,k)(r+a)^m z^r.
	\label{eq:pmk-definition}
\end{equation}
The formulas below show that \(P_{m,k}(z,a)\) extends to a rational
function of \(z\), defined for \(z\neq1\). We call these functions
weighted Lerch rational functions.
	\begin{theorem}
		\label{thm:pmk-egf}
		For \(|z|e^{|t|}<1\), the exponential generating function of \(P_{m,k}(z,a)\) with respect
		to \(m\) is
		\begin{align}
			\sum_{m=0}^{\infty}
			P_{m,k}(z,a)\frac{t^m}{m!}
			&=
			e^{at}
			\left[
			\alpha_k
			\frac{ze^t(1+ze^t)}{(1-ze^t)^3}
			+
			\beta_k
			\frac{ze^t}{(1-ze^t)^2}
			+
			\frac{1}{1-ze^t}
			\right].
			\label{eq:pmk-egf}
		\end{align}
	\end{theorem}
	
	\begin{proof}
		By \eqref{eq:pmk-definition},
		\begin{align*}
			\sum_{m=0}^{\infty}
			P_{m,k}(z,a)\frac{t^m}{m!}
			&=
			\sum_{m=0}^{\infty}
			\sum_{r=0}^{\infty}
			A(r+1,k)(r+a)^m z^r\frac{t^m}{m!}.
		\end{align*}
		The condition \(|z|e^{|t|}<1\) guarantees absolute convergence, so
		the order of summation may be interchanged. Hence,
		\begin{align*}
			\sum_{m=0}^{\infty}
			P_{m,k}(z,a)\frac{t^m}{m!}
			&=
			\sum_{r=0}^{\infty}
			A(r+1,k)z^r
			\sum_{m=0}^{\infty}
			\frac{((r+a)t)^m}{m!}\\
			&=
			e^{at}
			\sum_{r=0}^{\infty}
			A(r+1,k)(ze^t)^r.
		\end{align*}
		Using
		\[
		A(r+1,k)=\alpha_kr^2+\beta_kr+1
		\]
		and the standard geometric-series identities with \(q=ze^t\) gives
		\eqref{eq:pmk-egf}.
	\end{proof}
	\begin{corollary}
		For \(m=0\),
		\begin{equation}
			P_{0,k}(z,a)
			=
			\alpha_k
			\frac{z(1+z)}{(1-z)^3}
			+
			\beta_k
			\frac{z}{(1-z)^2}
			+
			\frac{1}{1-z}.
			\label{eq:p0k}
		\end{equation}
	\end{corollary}
	
	\begin{proof}
		Set \(t=0\) in Theorem~\ref{thm:pmk-egf}.
	\end{proof}

	\begin{theorem}\label{thm:pmk-relations}
		For \(m\in\mathbb{N}\),
		\begin{equation}
			\frac{\partial}{\partial a}P_{m,k}(z,a)
			=
			mP_{m-1,k}(z,a).
			\label{eq:pmk-a-derivative}
		\end{equation}
		
		For \(m\in\mathbb{N}_0\),
		\begin{equation}
			P_{m+1,k}(z,a)
			=
			(D+a)P_{m,k}(z,a),
			\label{eq:pmk-recursion}
		\end{equation}
		and
		\begin{equation}
			P_{m,k}(z,a)
			=
			(D+a)^mP_{0,k}(z,a).
			\label{eq:pmk-operator-power}
		\end{equation}
	\end{theorem}
	\begin{proof}
		By \eqref{eq:pmk-definition},
		\begin{equation*}
			P_{m,k}(z,a)=\sum_{r=0}^{\infty}A(r+1,k)(r+a)^m z^r, \ |z|<1.
		\end{equation*}
		
		Let \(K_z\) be a compact subset of \(\{z\in\mathbb{C}:|z|<1\},\) and let \(K_a\) be a compact subset of \(\mathbb{C}\). There exists \(\rho<1\) such that \(|z|\leq\rho\) for every \(z\in K_z\).
		Since \(A(r+1,k)\) is quadratic in \(r\), the defining series for \(P_{m,k}(z,a)\), as well as the series obtained after differentiation with respect to \(a\) or application of \(D\), is dominated on
		\(K_z\times K_a\) by a convergent series of the form \(C(1+r)^N\rho^r.\) Hence, these series converge uniformly on compact subsets, and the required operations may be performed term by term.
		
		Let \(m\in\mathbb{N}\). Differentiating with respect to \(a\), we obtain
		\begin{align*}
			\frac{\partial}{\partial a}P_{m,k}(z,a)
			&=\frac{\partial}{\partial a}\sum_{r=0}^{\infty}A(r+1,k)(r+a)^m z^r\\
			&=\sum_{r=0}^{\infty}A(r+1,k)\frac{\partial}{\partial a}(r+a)^m z^r\\
			&=m\sum_{r=0}^{\infty}A(r+1,k)(r+a)^{m-1}z^r\\
			&=mP_{m-1,k}(z,a).
		\end{align*}
		This proves \eqref{eq:pmk-a-derivative}. Next, let \(m\in\mathbb{N}_0\).  Applying the Euler operator \(D\) defined in \eqref{eq:EulerOperator} to \(P_{m,k}(z,a)\), we obtain
		\begin{align*}
			DP_{m,k}(z,a)
			&=
			z\frac{d}{dz}
			\sum_{r=0}^{\infty}
			A(r+1,k)(r+a)^m z^r\\
			&=
			\sum_{r=0}^{\infty}
			rA(r+1,k)(r+a)^m z^r\\
			&=\sum_{r=0}^{\infty}
			\bigl[(r+a)-a\bigr]
			A(r+1,k)(r+a)^m z^r\\
			&=\sum_{r=0}^{\infty}A(r+1,k)(r+a)^{m+1}z^r-a\sum_{r=0}^{\infty}A(r+1,k)(r+a)^m z^r\\
			&=P_{m+1,k}(z,a)-aP_{m,k}(z,a).
		\end{align*}
		Rearranging the terms proves \eqref{eq:pmk-recursion}. 
		
		It remains to prove \eqref{eq:pmk-operator-power}. We use induction on \(m\).  Let \(m=0\). Then \(P_{0,k}(z,a)
		=(D+a)^0P_{0,k}(z,a)\), so the formula holds.  Assume that 
		\begin{equation*}
			P_{m,k}(z,a)=(D+a)^mP_{0,k}(z,a)
		\end{equation*}
		for some \(m\in\mathbb{N}_0\).  By \eqref{eq:pmk-recursion},
		\begin{align*}
			P_{m+1,k}(z,a)
			&=
			(D+a)P_{m,k}(z,a)\\
			&=
			(D+a)(D+a)^mP_{0,k}(z,a)\\
			&=
			(D+a)^{m+1}P_{0,k}(z,a).
		\end{align*}
		Hence, the formula holds for \(m+1\). By the principle of mathematical induction, \eqref{eq:pmk-operator-power} holds for every	\(m\in\mathbb{N}_0\).
		
		Lastly, the identities are first obtained for \(|z|<1\). Since \(P_{m,k}(z,a)\) extends to a rational function of \(z\), they remain valid 	for all \(z\neq1\) by continuation.
	\end{proof}
	As an immediate consequence of Theorem \ref{thm:pmk-relations},
	\[
	P_{1,k}(z,a)
	=
	(D+a)P_{0,k}(z,a).
	\]
	
	\begin{corollary}
		For every \(m\in\mathbb{N}_0\), there exists a polynomial
		\(R_{m,k}(z,a)\) such that
		\begin{equation}
			P_{m,k}(z,a)
			=
			\frac{
				R_{m,k}(z,a)
			}{
				(1-z)^{m+3}
			}.
			\label{eq:pmk-rational}
		\end{equation}
	\end{corollary}
	
	\begin{proof}
		The function \(P_{0,k}(z,a)\) has denominator at most
		\((1-z)^3\). By \eqref{eq:pmk-operator-power},
		\[
		P_{m,k}(z,a)
		=
		(D+a)^mP_{0,k}(z,a).
		\]
		Each application of \(\displaystyle D=z\frac{d}{dz}\)
		increases the possible order of the pole at \(z=1\) by at most one,
		while multiplication by \(a\) does not change the denominator.
		Therefore, \(P_{m,k}(z,a)\) has denominator at most
		\((1-z)^{m+3}\).
	\end{proof}
	
	For \(n\in\mathbb{N}_0\), let \(\mathcal{A}_n(z)\) denote the Eulerian
	polynomial, defined by
	\[
	\mathcal{A}_0(z)=1
	\]
	and
	\[
	\mathcal{A}_n(z)
	=
	\sum_{j=0}^{n-1}
	\left\langle {n\atop j}\right\rangle z^j,
	\
	n\geq1,
	\]
	where \(\displaystyle\left\langle {n\atop j}\right\rangle\)	is an Eulerian number. The negative-order polylogarithm satisfies
	\begin{equation}
		\operatorname{Li}_{-n}(z)
		=
		\frac{z\mathcal{A}_n(z)}{(1-z)^{n+1}},
		\
		n\in\mathbb{N}_0.
		\label{eq:negative-polylog-eulerian}
	\end{equation}
	This identity follows from the rational form of the negative-order
	polylogarithm \cite{NIST2026}. Since
	\[
	\Phi(z,-n,1)
	=
	\frac{\operatorname{Li}_{-n}(z)}{z},
	\]
	we obtain
	\begin{equation}
		\Phi(z,-n,1)
		=
		\frac{\mathcal{A}_n(z)}{(1-z)^{n+1}}.
		\label{eq:lerch-eulerian}
	\end{equation}
	
	\begin{theorem}
		\label{thm:hk-eulerian}
		For \(m\in\mathbb{N}_0\),
		\begin{align}
			H_k(z,-m,1)
			&=
			\alpha_k
			\frac{\mathcal{A}_{m+2}(z)}{(1-z)^{m+3}}
			+
			(\beta_k-2\alpha_k)
			\frac{\mathcal{A}_{m+1}(z)}{(1-z)^{m+2}}
			\notag\\
			&\quad+
			(1-\beta_k+\alpha_k)
			\frac{\mathcal{A}_m(z)}{(1-z)^{m+1}}.
			\label{eq:hk-eulerian}
		\end{align}
		The identity holds for \(|z|<1\) and extends as a rational identity for
		\(z\neq1\).
	\end{theorem}
	
	\begin{proof}
		Set \(a=1\) and \(s=-m\) in
		\eqref{eq:hk-reduction-main}. This gives
		\begin{align*}
			H_k(z,-m,1)
			&=
			\alpha_k\Phi(z,-m-2,1)
			+
			(\beta_k-2\alpha_k)\Phi(z,-m-1,1)
			\\
			&\quad+
			(1-\beta_k+\alpha_k)\Phi(z,-m,1).
		\end{align*}
		Applying \eqref{eq:lerch-eulerian} to the three terms gives
		\eqref{eq:hk-eulerian}.
	\end{proof}
	
	\begin{corollary}
		For \(m\in\mathbb{N}_0\),
		\begin{equation}
			H_k(z,-m,1)
			=
			\frac{R_{m,k}(z,1)}{(1-z)^{m+3}},
			\label{eq:hk-eulerian-numerator}
		\end{equation}
		where
		\begin{align}
			R_{m,k}(z,1)
			&=
			\alpha_k\mathcal{A}_{m+2}(z)
			+
			(\beta_k-2\alpha_k)
			(1-z)\mathcal{A}_{m+1}(z)
			\notag\\
			&\quad+
			(1-\beta_k+\alpha_k)
			(1-z)^2\mathcal{A}_m(z).
			\label{eq:rmk-eulerian}
		\end{align}
	\end{corollary}
	
	\begin{proof}
		Write the three terms in \eqref{eq:hk-eulerian} over the common denominator
		\((1-z)^{m+3}\).
	\end{proof}
	
	The Euler polynomials \(E_n(a)\) are defined by
	\begin{equation}
		\frac{2e^{at}}{e^t+1}
		=
		\sum_{n=0}^{\infty}
		E_n(a)\frac{t^n}{n!},
		\
		|t|<\pi.
		\label{eq:euler-polynomial-generating-function}
	\end{equation}
	See \cite{NIST2026}.
	
	For \(\operatorname{Re}(s)>1\), separating the even and odd terms gives
	\begin{equation}
		\Phi(-1,s,a)
		=
		2^{-s}
		\left[
		\zeta\left(s,\frac{a}{2}\right)
		-
		\zeta\left(s,\frac{a+1}{2}\right)
		\right].
		\label{eq:alternating-lerch-hurwitz}
	\end{equation}
	Both sides extend by analytic continuation. Setting \(s=-n\) and using
	\eqref{eq:zeta-bernoulli}, we obtain
	\[
	\Phi(-1,-n,a)
	=
	\frac{2^n}{n+1}
	\left[
	B_{n+1}\left(\frac{a+1}{2}\right)
	-
	B_{n+1}\left(\frac{a}{2}\right)
	\right].
	\]
	The relation
	\[
	E_n(a)
	=
	\frac{2^{n+1}}{n+1}
	\left[
	B_{n+1}\left(\frac{a+1}{2}\right)
	-
	B_{n+1}\left(\frac{a}{2}\right)
	\right]
	\]
	therefore gives
	\begin{equation}
		\Phi(-1,-n,a)
		=
		\frac{1}{2}E_n(a),
		\qquad
		n\in\mathbb{N}_0.
		\label{eq:lerch-euler-polynomial}
	\end{equation}
	
	\begin{theorem}
		\label{thm:hk-euler-polynomial}
		Let \(m\in\mathbb{N}_0\). The continuation of \(H_k(-1,s,a)\) to
		\(s=-m\) satisfies
		\begin{align}
			H_k(-1,-m,a)
			&=
			\frac{1}{2}\alpha_kE_{m+2}(a)
			+
			\frac{1}{2}
			(\beta_k-2a\alpha_k)E_{m+1}(a)
			\notag\\
			&\quad+
			\frac{1}{2}
			(1-a\beta_k+a^2\alpha_k)E_m(a).
			\label{eq:hk-euler-polynomial}
		\end{align}
	\end{theorem}
	
	\begin{proof}
		Set \(z=-1\) and \(s=-m\) in the continued form of
		\eqref{eq:hk-reduction-main}. We obtain
		\begin{align*}
			H_k(-1,-m,a)
			&=
			\alpha_k\Phi(-1,-m-2,a)
			\\
			&\quad+
			(\beta_k-2a\alpha_k)
			\Phi(-1,-m-1,a)
			\\
			&\quad+
			(1-a\beta_k+a^2\alpha_k)
			\Phi(-1,-m,a).
		\end{align*}
		Applying \eqref{eq:lerch-euler-polynomial} to the three terms gives
		\eqref{eq:hk-euler-polynomial}.
	\end{proof}
	
	\begin{example}
		Taking \(m=0\) in Theorem~\ref{thm:hk-euler-polynomial} and using
		\[
		E_0(a)=1,
		\
		E_1(a)=a-\frac12,
		\ \text{and} \
		E_2(a)=a^2-a,
		\]
		we obtain
		\begin{align*}
			H_k(-1,0,a)
			&=
			\frac12
			\left[
			\alpha_kE_2(a)
			+
			(\beta_k-2a\alpha_k)E_1(a)
			+
			(1-a\beta_k+a^2\alpha_k)E_0(a)
			\right]
			\\
			&=
			\frac12-\frac{\beta_k}{4}
			=
			\frac12-\frac{3}{8}2^k.
		\end{align*}
		Thus, this continued value is independent of \(a\).
	\end{example}
	
	The special values obtained in this section connect the weighted
	Hurwitz--Lerch family with several classical polynomial sequences. The values
	at \(z=1\) are expressed through Bernoulli polynomials. For \(a=1\), the
	rational functions \(H_k(z,-m,1)\) have numerator polynomials given explicitly
	by Eulerian polynomials. The alternating values \(H_k(-1,-m,a)\) are expressed
	through Euler polynomials.
	
	\section{Conclusion}
	
	We introduced a Hurwitz--Lerch type family weighted by the
	\(k\)-augmented centered triangular numbers. Its dependence on \(k\) separates
	into the geometric factors \(1\), \(2\), and \(4\). This gives an inversion
	formula recovering \(\Phi(z,s,a)\), \(D\Phi(z,s,a)\), and \(D^2\Phi(z,s,a)\)	
	from three consecutive members of the family. The same Vandermonde argument
	applies to polynomial weights with coefficients depending on distinct
	geometric sequences.
	

\end{document}